\documentclass{article}
\usepackage[utf8]{inputenc}
\usepackage{amsmath,amsthm,amssymb,amsfonts}
\usepackage{graphicx}
\usepackage{todonotes}
\usepackage{tikz} 
\usepackage{comment}

\newtheorem{thm}{Theorem}

\newtheorem{lem}[thm]{Lemma}

\newcommand{\N}{{\mathbb N}}

\def\1{{\bf 1}}
\def\N{\mathbb N}
\def\ee{\epsilon}

\title{Null recurrence and transience for a binomial catastrophe model in random environment
}
\begin{document}

\maketitle

Luiz Renato Fontes$^1$, Fabio P. Machado \footnote{Instituto de Matemática e Estatística, Universidade de São Paulo, Rua do Matão 1010, 05508-090 São Paulo SP Brasil
E-mail: lrfontes@usp.br and fmachado@ime.usp.br} and Rinaldo B. Schinazi\footnote{Department of Mathematics, University of Colorado, Colorado Springs, CO 80933-7150, USA.
E-mail: rinaldo.schinazi@uccs.edu}

\begin{abstract}
We consider a discrete time population model for which each individual alive at time $n$ survives independently of everybody else at time $n+1$ with probability $\beta_n$. The sequence $(\beta_n)$ is i.i.d. and constitutes our random environment. Moreover, at every time $n$ we add $Z_n$ individuals to the population. The sequence $(Z_n)$ is also i.i.d. We find sufficient conditions for null recurrence and transience (positive recurrence has been addressed by Neuts in \cite{Neuts}). We apply our results to a particular $(Z_n)$ distribution and deterministic $\beta$. This particular case shows a rather unusual phase transition in $\beta$ in  the sense that the Markov chain goes from transience to null recurrence without ever reaching positive recurrence.
\end{abstract}
\section{The model and the results}

Consider a sequence of independent identically distributed (i.i.d. in short) random vectors, $(\beta_1,Z_1),(\beta_2,Z_2),\dots$ The $Z$ distribution is discrete with support on the set of natural numbers $\N$.
The $\beta$ distribution is continuous or discrete with support on $(0,1)$.
Assume that,
\begin{equation}
\label{eq:logmoment} 
  \mu=E(-\ln \beta)\in (0,+\infty).  
\end{equation}

 To ensure that this Markov chain is irreducible we also assume that 
 $$P(\beta\in (0,1))=1\mbox{ and }P(Z\geq 2)>0.$$

For $n\geq 1$, let ${\cal F}_n$ be the $\sigma$-algebra by $(\beta_1,Z_1)\dots (\beta_n,Z_n).$

Let $(X_n)$ be a Markov chain defined as follows. Set $X_0=1$ and $B_0=0$. For $n\geq 1$,
$$X_n=B_{n-1}+Z_n,$$
where the conditional distribution of $B_n$ given ${\cal F}_{n+1}$ is distributed according to a binomial distribution with parameters $X_n$ and $\beta_{n+1}$. 

We may think of $X_n$ as the size of a population at time $n$. At any time $n$ we sample a $\beta_{n+1}$ (the random environment) from a fixed distribution. Each individual alive at time $n$ survives to time $n+1$ independently of everybody else with probability $\beta_{n+1}$ (the binomial catastrophe). The population $X_{n+1}$ at time $n+1$ is made up of the survivors from time $n$  and of an influx of $Z_n$ new individuals.

The next two results state sufficient conditions for recurrence and transience under additional hypotheses.

\medskip

{\bf Hypothesis 1.} If $\beta_1$ and $Z_1$ are independent then the sequences $(\beta_i)_{i\geq 1}$ and $(Z_i)_{i\geq 1}$ are independent.

\begin{thm}
Assume that Hypothesis 1 holds. 
If 
\begin{equation}
    \limsup_{t\to\infty} t P(\ln Z>t)<\mu
\end{equation}
then the Markov chain $(X_n)$ is recurrent.
\end{thm}

\medskip

{\bf Hypothesis 2.} Assume that $E(\beta^{-\theta})<+\infty$ for some $\theta>0$.

\begin{thm}
Assume that Hypothesis 2 holds.
If \begin{equation}
    \liminf_{t\to\infty} t P(\ln Z>t)>\mu
\end{equation}
then the Markov chain $(X_n)$ is transient.
\end{thm}

We now give a necessary and sufficient condition for positive recurrence.

\begin{thm}
The Markov chain $(X_n)$ is positive recurrent if and only if 
$$E(\ln Z)<+\infty.$$
\end{thm}

Even when the chain is positive recurrent the paths of the chain can have huge fluctuations. See the simulation in Figure 1.

\section{An example}

For this example we will assume that $\beta$ has a point mass distribution. This constant is also denoted by $\beta$. Let the distribution of $Z$ be given by
$$P(Z=k)=\frac{C}{k(\ln k)^{a+1}},$$
for all $k\geq 2$. The parameter $a$ is strictly positive and $C>0$ only depends on $a$. The Markov chain $(X_n)$ exhibits three different behaviors.

\begin{itemize}
    \item If $a>1$ the chain $(X_n)$ is positive recurrent for all $\beta$ in $(0,1)$.
    \item If $a<1$ the chain $(X_n)$ is transient for all $\beta$ in $(0,1)$.
    \item If $a=1$ then there exists a critical value $\beta_c\in (0,1)$ such that
    if $\beta<\beta_c$ then $(X_n)$ is transient while if $\beta>\beta_c$ the chain is null recurrent.
\end{itemize}

Note that the case $a=1$ gives a (rare?) example of a phase transition from transience to null recurrence without ever reaching positive recurrence. This example requires $E(\ln Z)=+\infty$ which seems like an extreme hypothesis for a population model. However, the main point of this example is conceptual. Based on classical examples one may wrongly believe that null recurrence is a fleeting phenomenon. For instance, the simple random walk exhibits null recurrence for a single point in the parameter space. For our model, however, null recurrence can occur for a whole interval in the parameter space.

We now justify our claims. An easy comparison with an integral shows that
$$\lim_{t\to\infty}\frac{t P(\ln Z>t)}{C/(at^{a-1})}=1.$$
Using this limit with Theorems 1 and 2 yields the behavior of the chain for $a<1$ and $a>1$. 

Turning to $a=1$ we see that $E(\ln Z)=\infty$. In this case the chain is positive recurrent for no $\beta$ in $(0,1)$. The limit above becomes  
$$\lim_{t\to\infty}t P(\ln Z>t)=C.$$
Let $\beta_c=\exp(-C)$. Since $\beta$ is deterministic $\mu=-\ln \beta$.
By Theorem 1, 
if $\beta<\beta_c$ the chain is recurrent and therefore null recurrent. On the other hand if $\beta>\beta_c$ by Theorem 2 the chain is transient. 

\section{Literature}

Catastrophe models of the type we study in this paper go back to at least the 1980's, see for instance \cite{Brockwell} and \cite{Brockwell et al.}. These models are also related to branching processes with immigration, see for instance \cite{Heyde}, \cite{Pakes} and \cite{Seneta}.

Neuts in \cite{Neuts} introduced several models for catastrophes including the following Markov chain $(Y_n)$ which is closely related to our model. Let $p$ and $\beta$ be parameters in $(0,1)$. Let $(Z_n)$ be a sequence of i.i.d. random variables with support on the positive integers.
Given $Y_0\geq 0$, for $n\geq 1$,
\begin{itemize}
    \item With probability $p$, $Y_n=B_n$, where given $Y_{n-1}$, $B_n$ is a binomial random variable with parameters  $Y_{n-1}$ and $\beta$.
    \item With probability $1-p$, $Y_n=Y_{n-1}+Z_n$.
\end{itemize}
A necessary and sufficient condition for positive recurrence is given in \cite{Neuts}. One can show that this condition is equivalent to our Theorem 3 in the particular case when $\beta$ is a fixed constant. In fact, in the last section of this paper we will provide a coupling showing that Theorems 1, 2 and 3 hold true for Neut's model.

Recently, the chain $(Y_n)$ has been studied in \cite{Ben-Ari} and \cite{GH}. These articles concentrate mainly on properties of the stationary probability distribution in the positive recurrent case. We believe the present paper to be the first to address null recurrence and transience for binomial catastrophes. 

\section{Proofs}

We first give a construction of the process.
Let $\{\xi_{i,n}: n\geq 0,i\geq 1\}$ be a collection of independent Bernoulli random variables such that for all $n\geq 0$ and $i\geq 1$, $P(\xi_{i,n}=1)=\beta_n$. Let $B_0=0$ and given $X_1,\dots,X_n$, let
$$B_n=\sum_{i=1}^{X_n}\xi_{i,n+1}.$$
Then,
\begin{equation*}
  X_{n+1}=B_n+Z_{n+1}  
\end{equation*}

Next we give a useful representation formula for the distribution of $(X_n)$.

\begin{lem}
\label{lem:rep}
 For $n\geq 1$, 
 $$X_n\overset{d}{=}\sum_{i=1}^n B_{i,n},$$
 where given ${\cal F}_n$, $B_{i,n}$ is a binomial random variable with parameters $Z_i$ and $\prod_{j=i+1}^n \beta_j$ for $1\leq i\leq n$. Moreover, for fixed $n\geq 1$, $B_{1,n},\dots,B_{n,n}$ are independent.
\end{lem}
By convention we set $\prod_{j=n+1}^n \beta_j=1$. Therefore, $B_{n,n}=Z_n$.

\begin{proof}

We prove the Lemma by induction on $n$. Recall that $X_1=B_0+Z_1$, since $B_0=0$, $X_1=Z_1$ which has the same distribution as $B_{1,1}$. Hence the Lemma holds for $n=1$. Assume now that the Lemma holds for $n=k$. By construction,
\begin{align*}
X_{k+1}=&B_k+Z_{k+1}\\
=&\sum_{i=1}^{X_k}\xi_{i,k+1}+Z_{k+1},
\end{align*}
with $X_k$ independent of $(\xi_{i,k+1})_{i\geq 1}$. Thus,
 $B_{1,k},\dots,B_{k,k}$ are independent and independent of $Z_{k+1},\xi_{1,k+1},\xi_{2,k+1},\dots$
Let $S_{0,k}=0$ and $S_{j,k}=\sum_{i=1}^j B_{i,k}$ for $1\leq j\leq k$. Note that 
$$X_{k+1}\overset{d}{=}\sum_{i=1}^{S_{k,k}}\xi_{i,k+1}+Z_{k+1}.$$
Define,
$$\tilde B_{\ell,k}=\sum_{i=S_{\ell-1,k}+1}^{S_{\ell,k}} \xi_{i,k}.$$
Hence,
$$\sum_{i=1}^{S_{k,k}}\xi_{i,k}=\sum_{\ell=1}^k \tilde B_{\ell,k},$$
where $\tilde B_{1,k},\dots, \tilde B_{k,k}$ are independent. Moreover,
$$\tilde B_{\ell,k}\overset{d}{=}\sum_{i=1}^{B_{\ell,k}}\xi_{i,k+1}.$$

Therefore,
\begin{align*}
  X_{k+1} \overset{d}{=}& \sum_{i=1}^{S_{k,k}}\xi_{i,k}+Z_{k+1}\\
  \overset{d}{=}& \sum_{\ell=1}^k \tilde B_{\ell,k}+Z_{k+1}.\\
\end{align*}
To conclude the proof of the Lemma we show that $\tilde B_{\ell,k}\overset{d}{=}B_{\ell,k+1}.$ This comes from the following elementary fact. Assume that $X$ is a binomial random variable with parameters $N$ and $p$ and that $X$ is independent of the i.i.d. sequence $(\xi_i)_{i\geq 1}$ of Bernoulli random variables with parameter $f$. Then,
$$\sum_{i=1}^X \xi_i$$
is a binomial random variable with parameters $N$ and $pf$.

\end{proof}

We will prove Theorems 1 and 2 by using the following classical recurrence criterion for Markov chains. State $1$ is recurrent for the  Markov chain $(X_n)$ if and only if
$$\sum_{n\geq 1}P(X_n=1|X_0=1)=+\infty.$$
Let $p_1=P(Z=1)$ and assume for simplicity that $p_1>0$. Using Lemma \ref{lem:rep},
\begin{align*}
    P(X_n=1|{\cal F}_n)=&\prod_{i=1}^n P(B_{i,n}=0)\\
    =&p_1\prod_{i=1}^{n-1}\left(1-\prod_{j=i+1}^n\beta_j\right)^{Z_i}\\
    =&p_1\exp\left(\sum_{i=1}^{n-1}Z_i\ln\left(1-\prod_{j=i+1}^n\beta_j\right)\right)\\
\end{align*}
 Using now that the sequence $((Z_n,\beta_n))_{n\geq 0}$ is exchangeable,
\begin{align*}
\exp\left(\sum_{i=1}^{n-1}Z_i\ln\left(1-\prod_{j=i+1}^n\beta_j\right)\right)&\overset{d}{=}
\exp\left(\sum_{i=1}^{n-1}Z_{n-i}\ln\left(1-\prod_{j=i+1}^n\beta_{n-j}\right)\right)
\end{align*}
We do the changes of indices, $n-i=\ell$ and $n-j=m$ to get,
\begin{align*}
\exp\left(\sum_{i=1}^{n-1}Z_i\ln\left(1-\prod_{j=i+1}^n\beta_j\right)\right)&\overset{d}{=}
\exp\left(\sum_{\ell=1}^{n-1}Z_\ell\ln\left(1-\prod_{m=0}^{\ell-1}\beta_j\right)\right)\\
&\overset{d}{=}\exp\left(\sum_{\ell=1}^{n-1}Z_{\ell+1}\ln\left(1-\prod_{m=1}^\ell\beta_j\right)\right).
\end{align*}

For $n\geq 1$, let
$$I_n=E\left [\exp\left(\sum_{i=1}^{n}Z_{i+1}\ln\left(1-\prod_{j=1}^i\beta_j\right)\right)\right].$$
Then,
\begin{equation}
\label{eq:I}
    P(X_n=1)=p_1I_{n-1},
\end{equation}

\subsection{Proof of Theorem 1}

\begin{proof}

For $j\geq 1$, let $W_j=-\ln\beta_j$ and 
$$\overline W_i=\frac{1}{i}\sum_{i=1}^j W_i.$$
Recall that $E(W_1)=\mu<+\infty$. Hence, for every $\ee$ in $(0,\mu)$, there exists $\alpha>0$ such that for all $i\geq 1$,
\begin{equation}
    P(\overline W_i<\mu-\ee)\leq e^{-\alpha i}
\end{equation}
Define the event
$$A_N=\{\overline W_i\geq \mu-\ee\,\forall i>N\}.$$
Choose a natural number $N$ such that $P(A_N)>0$. Let
$$S_n=\sum_{i=1}^{N}Z_{i+1}\ln\left(1-\prod_{j=1}^i\beta_j\right).$$

From (\ref{eq:I}) we have for $n>N$,
\begin{align*}
    I_n\geq&E\left [\exp\left(S_N\right)
    \exp\left(\sum_{i=N+1}^{n}Z_{i+1}\ln\left(1-e^{-i(\mu-\ee)}\right)\right);A_N\right ].
\end{align*}
By Hypothesis 1, $\sum_{i=N+1}^{n}Z_{i+1}\ln\left(1-e^{-i(\mu-\ee)}\right)$ is independent of $A_N$ and of $\sum_{i=1}^{N}Z_{i+1}\ln\left(1-\prod_{j=1}^i\beta_j\right)$. Hence,
\begin{align*}
&E\left [\exp\left(S_N\right)
    \exp\left(\sum_{i=N+1}^{n}Z_{i+1}\ln\left(1-e^{-i(\mu-\ee)}\right)\right);A_N\right ]=\\
 &E\left [\exp\left(S_N\right);A_N\right]
   E\left[ \exp\left(\sum_{i=N+1}^{n}Z_{i+1}\ln\left(1-e^{-i(\mu-\ee)}\right)\right)\right ].
\end{align*}
By the Harris-FKG inequality (see for instance (2.4) in \cite{Grimmett}),
\begin{align*}
E\left [\exp\left(S_N\right);A_N\right]\geq E\left [\exp\left(S_N\right)\right]P(A_N)
\end{align*}
For fixed $N$ the r.h.s. is just a constant $K(N)>0$. Hence,
$$I_n\geq K(N)E \left [\exp\left(\sum_{i=N+1}^{n}Z_{i+1}\ln\left(1-e^{-i(\mu-\ee)}\right)\right)\right ].$$

We now turn to,
\begin{align*}
   &E \left [\exp\left(\sum_{i=N+1}^{n}Z_{i+1}\ln\left(1-e^{-i(\mu-\ee)}\right)\right)\right ]
   \geq \\
   &E \left [\exp\left(-(1+\ee)\sum_{i=N+1}^{n}Z_{i+1}e^{-i(\mu-\ee)}\right)\right ]\geq \\
   &E \left [\exp\left(-(1+\ee)\sum_{i=1}^{n}Z_{i+1}e^{-i(\mu-\ee)}\right)\right ]
   \equiv L_n\\
\end{align*}
where we use that $\ln(1-x)\geq -(1+\ee)x$ for $x$ small enough. We take $N$ large enough so that
this inequality holds for $x=e^{-i(\mu-\ee)}$ for all $i>N$.

Hence, to show recurrence it is enough to show that the series with general term $L_n$ is infinite.
In  fact this will be a consequence of
$$\sum_{n\geq 1}\tilde L_n=+\infty,$$
where
\begin{align*}
   \tilde L_n=E(\exp(-\sum_{i=1}^n c^iZ_i))
   =\prod_{i=1}^n E(\exp(-c^iZ_i)), 
\end{align*}
where $0<c<1$.

To estimate $\tilde L_n$ we start with the following Lemma.

\begin{lem} 
\label{Laplace}
Let $Z$ be a random variable with support on the positive integers. Then, for $\lambda>0$,
$$E(e^{-\lambda Z})=1-(e^{\lambda}-1)\sum_{k\geq 1}e^{-\lambda k}P(Z\geq k).$$
\end{lem}

\begin{proof}
\begin{align*}
   E(e^{-\lambda Z})=&\sum_{k\geq 0}e^{-\lambda k}P(Z=k)\\ 
   =&\sum_{k\geq 0}e^{-\lambda k}\left(P(Z\geq k)-P(Z\geq k+1)\right)\\ 
   =&\sum_{k\geq 0}e^{-\lambda k}P(Z\geq k)-\sum_{k\geq 1}e^{-\lambda (k-1)}P(Z\geq k)\\ 
   =&1-(e^{\lambda}-1)\sum_{k\geq 1}e^{-\lambda k}P(Z\geq k).
\end{align*}
\end{proof}

Using Lemma \ref{Laplace} with $\lambda=c^i$,
\begin{align*}
    E(e^{-c^iZ_i})=&1-(e^{c^i}-1)\sum_{k\geq 1}e^{-c^i k}P(Z\geq k)\\
    \end{align*}

We now use that 
$$\limsup_{k\to\infty}kP(\ln Z>k)<\mu.$$
For $\epsilon>0$ small enough there is an integer $k_0$ such that for all $k\geq k_0$,
$$P(\ln Z\geq \ln k)\leq \frac{\mu-\epsilon}{\ln k}.$$
Hence,
\begin{equation}
\label{whole}
 E(e^{-c^iZ})\geq 1- (e^{c^i}-1)\sum_{k=1}^{k_0}e^{-c^i k}P(Z\geq k)
 -(\mu-\epsilon)(e^{c^i}-1)\sum_{k\geq k_0+1}\frac{1}{\ln k}e^{-c^i k}   
\end{equation}

We first bound the finite sum,
\begin{align*}
(e^{c^i}-1)\sum_{k=1}^{k_0}e^{-c^i k}P(Z\geq k)&\leq (e^{c^i}-1)\sum_{k=1}^{k_0}e^{-c^i k}\\
=& (e^{c^i}-1)e^{-c^i}\frac{1-e^{-c^ik_0}}{1-e^{-c^i}}\\
=&1-e^{-c^ik_0}
\end{align*}
Using this bound in equation (\ref{whole}),
\begin{equation}
 E(e^{-c^iZ})\geq 1-(1-e^{-c^ik_0}) -(\mu-\epsilon)(e^{c^i}-1)\sum_{k\geq k_0+1}\frac{1}{\ln k}e^{-c^i k}   
\end{equation}
Observe that for any $\delta>0$ there exists a $\gamma>0$ such that if $|x|<\gamma$ then
$$e^x-1\leq (1+\delta)x.$$
Hence, there exists $i_0$ such that if $i\geq i_0$ then
\begin{equation}
\label{delta}
 E(e^{-c^iZ})\geq 1-(1-e^{-c^ik_0})-(1+\delta)c^i\left((\mu-\epsilon)\sum_{k\geq k_0+1}\frac{1}{\ln k}e^{-c^i k}\right)   
\end{equation}    

The following elementary lemma estimates the tail of the series in equation (\ref{delta}).

\begin{lem}
\label{series}
Let $0<c<1$, $d=1/c$ and $i\geq 1$ then
$$c^i\sum_{k\geq d^i}\frac{1}{\ln k}e^{-c^i k}\leq \frac{k_1}{i\ln d},$$
where $k_1=\frac{e^{-1}}{1-e^{-1}}$.
\end{lem}

\begin{proof}

\begin{align*}
c^i\sum_{k\geq d^i}\frac{1}{\ln k}e^{-c^i k}=&c^i\sum_{\ell=1}^\infty \sum_{k=\ell d^i}^{(\ell+1) d^i-1}\frac{1}{\ln k}e^{-c^i k}\\
\leq& c^i\sum_{\ell=1}^\infty d^ie^{-c^i \ell d^i}\frac{1}{\ln (\ell d^i)}\\
\leq &\sum_{\ell=1}^\infty e^{- \ell }\frac{1}{i\ln d}\\
=&\frac{1}{i}\frac{e^{-1}}{(1-e^{-1})\ln d}.
\end{align*}

\end{proof}

Let $c=e^{-(\mu-\epsilon)}$ and therefore,
$\ln d=\mu-\epsilon.$
Using Lemma \ref{series} in equation (\ref{delta}),
\begin{align*}
E(e^{-c^iZ})\geq 1-(1-e^{-c^ik_0})-(1+\delta) \frac{k_1}{i}.    
\end{align*}
Thus,
\begin{align*}
    \tilde L_n=&\prod_{i=1}^n E(e^{-c^iZ_i})\\
    \geq &\prod_{i=1}^{i_0} E(e^{-c^iZ_i})\prod_{i=i_0+1}^n\left( 1-(1-e^{-c^ik_0})-(1+\delta) \frac{k_1}{i}\right)
\end{align*}
For $i_0$ large enough and $i\geq i_0$,
$$\ln\left( 1-(1-e^{-c^ik_0})-(1+\delta) \frac{k_1}{i}\right)\sim  -(1-e^{-c^ik_0})-(1+\delta) \frac{k_1}{i}.$$
Observe that the series with general term $(1-e^{-c^ik_0})$ converges and that
$$\sum_{i=1}^n \frac{1}{i}\sim \ln n.$$
Thus, for $n$ large enough
$$ \tilde L_n\geq \frac{k_2}{n^\alpha},$$
where $k_2$ is a strictly positive constant and
$$\alpha=(1+\delta)k_1.$$
By taking $\delta>0$ small enough we get $\alpha<1$.
Therefore, the series with general term $\tilde L_n$ diverges. We let $c=e^{-(\mu-\epsilon)}$ in $\tilde L_n$. By Jensen's inequality,
$$\tilde L_n^{1+\epsilon}\leq L_n.$$
Hence, the series with general term $L_n$ diverges provided $\epsilon>0$ is taken small enough. This completes the proof of Theorem 1.

\end{proof}

\subsection{Proof of Theorem 2}

Recall that for $n\geq 1$,
$$I_n=E\left [\exp\left(\sum_{i=1}^{n}Z_{i+1}\ln\left(1-\prod_{j=1}^i\beta_j\right)\right)\right].$$
and $P(X_n=1)=p_1I_{n-1}.$
Let
$$B_n=\{-\frac{1}{i}\sum_{i=1}^j \ln\beta_j\leq \mu+\ee\,\quad \forall i>K\ln n\},$$
for some fixed $K>0$.
Note that on $B_n$,
\begin{align*}
    \ln\left(1-\prod_{j=1}^i\beta_j\right)\leq &\ln\left(1-e^{-i(\mu+\ee)}\right)\\
    \leq & -e^{-i(\mu+\ee)}\\
\end{align*}

Therefore,
$$I_n\leq E\left[\exp\left(-\sum_{i=K\ln n}^nZ_{i+1}e^{-i(\mu+\ee)}\right)\right]+P(B^c_n).$$
It follows from Hypothesis 2 that for any $\ee>0$ there exists an $\alpha>0$ such that
for all $i\geq 1$,
$$P(\overline W_i>\mu+\ee)\leq e^{-\alpha i}.$$
Hence,
$$P(B^c_n)\leq \frac{1}{(1-e^{-\alpha})n^{\alpha K}}.$$
Thus, by taking $K$ such that $\alpha K>1$ the series with general term $P(B^c_n)$ converges. Therefore, to show transience it is enough to prove that
$$\sum_{n\geq 1} U_n<+\infty,$$
where
$$U_n=E\left[\exp\left(-\sum_{i=K\ln n}^nZ_{i+1}e^{-i(\mu+\ee)}\right)\right].$$
Using that the sequence $(Z_n)$ is i.i.d. 
$$U_n=\prod_{i=K\ln n}^nE\left[\exp\left(- e^{-i(\mu+\ee)}Z\right)\right].$$
Let
$$v_i= -\ln E[\exp(-\lambda_i Z)]$$
where
$\lambda_i= e^{-i(\mu+\ee)}$
and
$$V_n=\sum_{i=K\ln n}^n v_i.$$
Then, $U_n=\exp(-V_n)$. Using Lemma \ref{Laplace},
\begin{align*}
    v_i=&-\ln\left[1- (e^{\lambda_i}-1)\sum_{k\geq 1}e^{-\lambda_i k}P(Z\geq k)\right]\\
    \geq& (e^{\lambda_i}-1)\sum_{k\geq 1}e^{-\lambda_i k}P(Z\geq k)
\end{align*}
Let
$$\mu'=\frac{1}{2}\left(\liminf_{k}kP(\ln Z>k)+\mu\right).$$
Using that $\liminf_{k}kP(\ln Z>k)>\mu$, there exists $k_3$ such that for $k\geq k_3$,
$$P(\ln Z\geq \ln k)\geq \frac{\mu'}{\ln k}.$$
Let $B=e^{\mu+\epsilon}$ and $\delta>0$.
For $i$ large enough,
\begin{align*}
    v_i\geq \mu'\lambda_i\sum_{k\geq \delta B^i}\frac{e^{-\lambda_i k}}{\ln k}
\end{align*}
We estimate this series using an integral and the fact that $\lambda_i B^i=1$.
\begin{align*}
   v_i\geq& \mu'\lambda_i\int_{\delta B^i}^\infty \frac{e^{-\lambda_i x}}{\ln x}dx\\
    =&\mu'\lambda_i B^i\int_{\delta}^\infty \frac{e^{- x}}{\ln x+i\ln B}dx\\
    =&\frac{\mu'}{i\ln B}\int_{\delta}^\infty \frac{i\ln B }{\ln x+i\ln B}e^{- x}dx
\end{align*}
By the Monotone Convergence Theorem,
\begin{align*}
  \lim_{i\to\infty}  \int_{\delta}^\infty \frac{i\ln B }{\ln x+i\ln B}e^{- x}dx=&
  \int_{\delta}^\infty e^{- x}dx\\
  =& e^{- \delta}
\end{align*}
Let $N_1$ such that for $i\geq N_1$,
\begin{align*}
  \int_{\delta}^\infty \frac{i\ln B }{\ln x+i\ln B}e^{-(1-\epsilon) x}dx\geq &e^{- \delta}-\epsilon\\
  \geq& 1-\delta-\epsilon\\
  =&1-2\epsilon,
\end{align*}
where we let $\delta=\epsilon.$
Hence, for $i$ large enough,
\begin{align*}
 v_i\geq &\frac{\mu'}{i\ln B} (1-2\epsilon)\\
 =&\frac{\mu'}{\mu+\epsilon}(1-2\epsilon) \frac{1}{i}.
\end{align*}
Let
$$k_4=\frac{\mu'}{\mu+\epsilon}(1-2\epsilon).$$
Since $\mu'>\mu$, by taking $\epsilon>0$ small enough we get $k_4>1$.
Finally, for $n$ large enough
\begin{align*}
V_n=&\sum_{i=K\ln n}^n v_i\\  
\geq & \sum_{i=K\ln n}^n \frac{k_4}{i}\\
\geq  & k_4\left(\ln n-\ln(K\ln n)\right )-\epsilon
\end{align*}
Hence, 
\begin{align*}
    U_n\leq & \exp(-V_n)\\
    \leq& \exp\left(-k_4\left(\ln n-\ln(K\ln n)\right )+\epsilon\right)\\
    =&e^{\epsilon}(\frac{n}{K\ln n})^{-k_4}.
\end{align*}
This shows that the series with general term $U_n$ converges and completes the proof of Theorem 2.

\subsection{Proof of Theorem 3}

Define for $1\leq i\leq n$,
$$\beta_{i,n}=\prod_{j=i+1}^n \beta_j.$$
By convention $\beta_{n,n}$ is set to be 1.
Let $0<s<1$ and for $1\leq i\leq n$, 
$$s_i=1-(1-s)\beta_{i,n}.$$
Note that $s_n=s$. Recall that
$$X_n=B_{n-1}+Z_n.$$
Given ${\cal F}_n$, $B_{n-1}$ is a binomial random variable with parameters $X_{n-1}$ and $\beta_n$. Thus,
\begin{align*}
E(s_n^{X_n}|{\cal F}_n)=&s_n^{Z_n}E\left[\left(1-\beta_n\right)^{X_{n-1}}|{\cal F}_n\right]\\
=&s_n^{Z_n}E\left[s_{n-1}^{X_{n-1}}|{\cal F}_n\right]
\end{align*}

Iterating the preceding equality we get,
\begin{align*}
 E(s_n^{X_n}|{\cal F}_n)=&s_n^{Z_n}s_{n-1}^{Z_{n-1}}E\left[s_{n-2}^{X_{n-2}}|{\cal F}_n\right]\\
 =&s_n^{Z_n}s_{n-1}^{Z_{n-1}}\dots s_1^{Z_1}\\
\end{align*}
Hence,
\begin{align*}
    E(s^{X_n})=&E\left[\exp\left(\sum_{i=1}^n Z_i\ln\left(1-\beta_{i,n}(1-s)\right)\right)\right]\\
    =&E\left[\exp\left(\sum_{i=1}^n Z_i\ln\left(1-\beta_{1,i}(1-s)\right)\right)\right]
\end{align*}

We now prove the direct implication in Theorem 3. That is, if $E(\ln Z)<+\infty$ then the process $(X_n)$ is positive recurrent.

Note that $\beta_{1,i}$ converges to 0 a.s. as $i$ goes to infinity. Hence, to show convergence of $\sum_{i=1}^n Z_i\ln\left(1-\beta_{1,i}(1-s)\right)$ it is enough to show convergence of $S_n\equiv\sum_{i=1}^n Z_i\beta_{1,i}$. Let ${\cal F}_Z$ be the $\sigma$-algebra generated by the sequence $(Z_n)$. Then,
\begin{align*}
    E\left(S_n|{\cal F}_Z\right)=&\sum_{i=1}^n Z_iE(\beta_{1,i})\\
    =&\sum_{i=1}^n Z_iE(\beta)^{i-1}.
\end{align*}
Since $(S_n)$ is an increasing sequence,
\begin{align*}
\lim_{n\to\infty}E\left(S_n|{\cal F}_Z\right)=&E\left(\lim_{n\to\infty}S_n|{\cal F}_Z\right)\\
=&\sum_{i=1}^\infty Z_iE(\beta)^{i-1}
\end{align*}

Thus, in order to prove that $(S_n)$ converges a.s. we will show that the series $\sum_{i\geq 1}
Z_ib^{i-1}$ converges for any $0<b<1$. 

\begin{lem}
\label{logmoment}
Let $(Z_i)$ be an i.i.d. sequence with support in $\N$.  The series 
$$\sum_{i\geq 1}
Z_ib^{i}$$ converges almost surely for all $b$ in $(0,1)$ if and only if $E(\ln Z)<+\infty$.
\end{lem}

\begin{proof}

Assume that $E(\ln Z)<+\infty$. Then, for any constant $c>0$
$$\sum_{i=1}^\infty P(\ln Z_i>ci)<\infty.$$
Hence, by Borel-Cantelli Lemma 
$$P(\ln Z_i>ci \mbox{ i.o.})=0.$$
Since
$$\sum_{i\geq 1}
Z_ib^{i}=\sum_{i\geq 1} \exp(\ln Z_i+i\ln b),$$
we see that $\sum_{i\geq 1}
Z_ib^{i}<+\infty$ for all $0<b<1$.

Conversely, assume that $E(\ln Z)=+\infty$. Then, for any $b$ in $(0,1)$,
$$\sum_{i\geq 1}P(\ln Z_i\geq -i\ln b)=+\infty.$$
Hence, by the second Borel-Cantelli Lemma,
$$P(\ln Z_i\geq -i\ln b\mbox{ i.o.})=1.$$
Thus, $\sum_{i=1}^n b^iZ_i=+\infty.$
This concludes the proof of the Lemma and of the direct implication in Theorem 3.

We now turn to the converse in Theorem 3. Assume that $E(\ln Z)=+\infty$. 
We will show that for $0<s<1$, $E(s^{X_n})$ converges to $0$. This is enough to show that $(X_n)$ is not positive recurrent.  
Recall that
\begin{align*}
    E(s^{X_n})=E\left[\exp\left(\sum_{i=1}^n Z_i\ln\left(1-\beta_{1,i}(1-s)\right)\right)\right]
\end{align*}
Note that 
$-\sum_{i=1}^n Z_i\ln\left(1-\beta_{1,i}(1-s)\right)=+\infty$ if and only if $\sum_{i=1}^n Z_i\beta_{1,i}=+\infty$.
For $i\geq 1$, let
$$W_i=\sum_{j=1}^i -\ln \beta_j.$$
For $N\geq 1$, let
$$B_N=\{W_i\leq i(\mu+1)\mbox{ for all }i\geq N\},$$
where $\mu=E(-\ln \beta)$. For $n\geq N$,
\begin{align*}
  \sum_{i=1}^n Z_i\beta_{1,i}&=\sum_{i=1}^n Z_i\exp(-W_i)\\
  &\geq  \sum_{i=N}^n Z_i\exp(-W_i)\\
  &\geq {\bf 1}_{B_N} \sum_{i=N}^n Z_i\exp(-i(\mu+1))
\end{align*}
Note that ${\bf 1}_{B_N}$ converges to 1 as $N$ goes to infinity. By Lemma \ref{logmoment} we know that
$$\sum_{i=1}^n b^iZ_i=+\infty,$$
for any $0<b<1$. Hence, $\sum_{i=1}^n Z_i\beta_{1,i}=+\infty$ a.s. and $E(s^{X_n})$ converges to $0$ as $n$ goes to infinity for any $0<s<1$.
This completes the proof of Theorem 3.
\end{proof}

\section{ Neuts' model}

In this section we show that Theorems 1, 2, and 3 are true for Neuts' model. Actually, the Markov chain $(Y_n)$ that we construct below generalizes Neut's model in that the sequence $(\beta_i)$ is i.i.d.  

Let $(G_n)$ be an i.i.d. sequence of random variables with support on $\N$ and such that
$P(G>k)=p^{k-1}$ for all $k\in\N$.
For $k\geq 1$, let
$$T_k=\sum_{i=1}^k G_i.$$
Let $(\beta_i)_{i\geq 1}$ be a sequence of i.i.d. random variables with support in $(0,1)$. For $i\geq 1$ let $(\xi_{i,j})_{j\geq 1}$ be a sequence of independent Bernoulli random variables with parameter $\beta_i$. All these sequences (for different $i$'s) are independent of each other. Let $(Z_n)$ be a sequence of i.i.d. random variables with support in $\N$. We define the process $(Y_n)$ as follows. By convention, in what follows any sum from $k$ to $\ell$ where $k>\ell$ is set to 0.

\begin{itemize}
    \item $Y_0=Z_0$.
    \item If $n=T_k$ for some $k\geq 1$ and if $Y_{n-1}\geq 1$,
    $$Y_n=\sum_{j=1}^{Y_{n-1}}\xi_{n,j}.$$
    \item If $n\not =T_k$ for all $k\geq 1$ then
    $$Y_n=Y_{n-1}+Z_n.$$
\end{itemize}

We now construct a Markov chain $(X_n')$ on the same probability space. For $n\geq 1$, let $T'_n=T_n-1$ and define
$$X_n'=Y_{T_n'}.$$
Let 
$$Z_n'=\sum_{k=T_{n-1}+1}^{T_n-1} Z_k.$$
Then,
\begin{equation*}
    X_n'=\sum_{j=1}^{X_{n-1}}\xi_{n,j}+Z_n'.
\end{equation*}

Note that $(Z_n')$ is an i.i.d. sequence for which $Z'$ has the same distribution as
$$\sum_{i=1}^{G'}Z_i,$$
where $P(G'\geq k)=p^k$ for $k\geq 0$. Observe now that the conditions on the distribution of $Z$ in Theorems 1, 2 and 3 are equivalent to the conditions on the distribution of $Z'$. Hence, Theorems 1, 2 and 3 apply to the chain $(X_n')$. Since $X_n'=Y_{T_n'}$, these theorems apply to 
 Neut's chain $(Y_n)$ as well.

\bibliographystyle{amsplain}

{\bf Acknowledgments} We thank two anonymous referees for their careful reading and constructive suggestions.
FM was partially supported by CNPq (303699/2018-3) and Fapesp (17/10555-0).
LRF was partially supported by CNPq (307884/2019-8) and Fapesp (2017/10555-0). R.B.S. visit to the University of S\~ao Paulo was supported by Fapesp (17/10555-0).
\begin{figure}[ht] 
\includegraphics[scale=0.5]{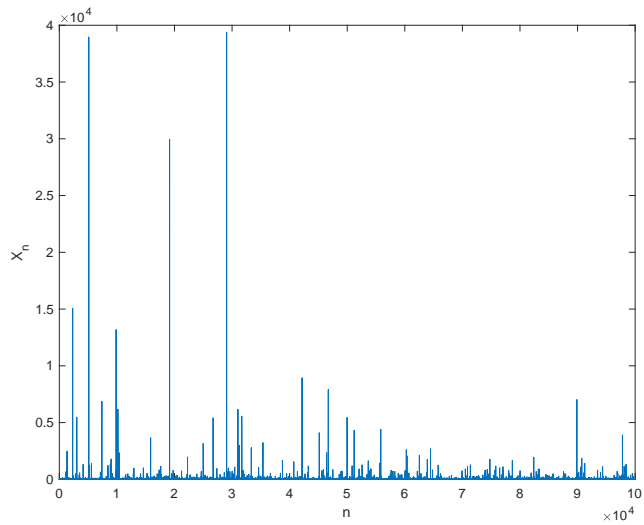}
\caption{ This is a simulation of the Markov chain $(X_n)$ for $1\leq n\leq 10^5$. The distribution of $\beta$ is uniform on $(0,1)$ and the distribution of $Z$ is given by $P(Z=k)=C/k^2$ for $k\geq 1$. Hence, $E(Z)=+\infty$ and $E(\ln Z)<+\infty$. The chain is positive recurrent. The average of $X_n$ over all $n$ in the range is 14.45. The maximum of $X_n$ is 39388.}
\end{figure} 

\end{document}